\documentclass[a4paper,12pt]{article}

\usepackage{hyperref}

\usepackage{fullpage}
\usepackage{graphicx}
\usepackage{amsmath}
\usepackage{amssymb}
\usepackage{subfigure}
\usepackage{epsfig}

\usepackage{caption,color}
\usepackage{amsthm}

\usepackage{bm}

\usepackage{url}

\usepackage{accents}

\theoremstyle{plain}
\newtheorem{thm}{Theorem}[section]

\theoremstyle{remark}
\newtheorem{rem}{Remark}[section]

\author{
Li-Chang Hung\footnote{Corresponding author's email address: \texttt{lichang.hung@gmail.com
} 
         }
\vspace{10mm}
\\
 \small
 \textit{Department of Mathematics, National Taiwan University, Republic of Taiwan}
\\
\\
\small
 \small 
}

\title{Blow-up in reaction-diffusion systems under Robin boundary conditions
}

\date{
}

\begin{document}
\maketitle

\begin{abstract}
In this paper we apply the differential inequality technique of Payne {\it et. al} \cite{Payne&SchaeferRobin08} to show that a reaction-diffusion system admits blow-up solutions, and to determine an upper bound for the blow-up time. For a particular nonlinearity, a lower bound on the blow-up time, when blow-up does occur, is also given. 
\end{abstract}




\section{Introduction}
\vspace{5mm}


Since the pioneering work of Fujita (\cite{Fujita-blowup-66}) on the blow-up of solutions of nonlinear diffusion equations, there has been considerable interest in the study of such solutions for nonlinear parabolic equations. Blow-up phenomena can be observed in nature and is important in various disciplines such as biology, chemistry and physics. More detailed results related to the blow-up of solutions can be seen, for example, in the monograph \cite{Quittner&Souplet07,Samarskii95,Straughan98}, the surveys \cite{Bandle&Brunner97,Fila&Filo-blow-survey96,Galaktionov&Vazquez02,Souplet-survey05} and bibliographies cited therein.


Recently, Payne and Schaefer (\cite{Payne&Schaefer08nonlineara,Payne&Philippin&Schaefer-blowup-bound08,Payne&Schaefer-blowup-lowerbdd-Neumann06,Payne&Schaefer07JMAA,Payne&SchaeferRobin08,Payne&Schaefer09ProRoy}) have applied the \textit{energy method} to derive a differential inequality for certain integrals corresponding to the blow-up solutions of parabolic equations. By means of this method, they find lower bounds on the blow-up time in certain nonlinear parabolic problems. In the present paper, we find that the above-mentioned method can be applied to \textit{reaction-diffusion systems} to obtain generalizations of results in \cite{Payne&SchaeferRobin08}. More precisely, it can be shown that Theorem 2.1 and Theorem 3.1 in \cite{Payne&SchaeferRobin08} can be extended to Theorem~\ref{thm: main thm} and  Theorem~\ref{thm: main thm for lower bound for blow-up time}, respectively in this paper.

Compared with cases of single equations, little work appears to have been devoted to systems. Motivated by Theorem 2.1 in \cite{Payne&SchaeferRobin08}, we show that the result in Theorem 2.1 can be extended to the \textit{gradient system} case. To be specific, we consider the following initial-boundary value problem for the reaction-diffusion system under a Robin boundary condition, i.e.
\begin{equation}\label{eqn: IBV prob for RD system}
\begin{cases}
\vspace{2mm}
u_t= \Delta u+f_1(u,v), \quad \text{in}\quad \Omega\times(0,\infty), \\
\vspace{2mm}
v_t= \Delta v+f_2(u,v), \quad \text{in}\quad \Omega\times(0,\infty),\\
\vspace{2mm}
\frac{\displaystyle\partial u}{\displaystyle\partial \nu}+\gamma_1\,u=0, \quad
\frac{\displaystyle\partial v}{\displaystyle\partial \nu}+\gamma_2\,v=0, \quad  \text{on}\quad \partial\Omega\times(0,\infty),\\
u(x,0)=g_1(x), \quad v(x,0)=g_2(x) \quad \text{in}\quad \bar{\Omega},
\end{cases}
\end{equation}
where $(u,v)=(u(x,t),v(x,t))$; $\Omega\in\mathbb{R}^N$, $N\geq2$ is a bounded domain with smooth $\partial\Omega$; $\gamma_1$ and $\gamma_2$ are positive constants; $\nu$ is the unit outward normal on $\partial\Omega$; and $g_1(x)$ and $g_2(x)$ are nonnegative functions which do not completely vanish. Suppose that the nonlinearity $f_1(u,v)$ and $f_2(u,v)$ in \eqref{eqn: IBV prob for RD system} satisfy

\begin{itemize}
  \item [(H1)] $u\,f_1(u,v)+v\,f_2(u,v)\geq 2(1+\alpha)\,F(u,v)$,
\end{itemize}
where $F(u,v)$ is a solution to $\partial_u F(u,v)=f_1(u,v)$, $\partial_v F(u,v)=f_2(u,v)$, and $\alpha>0$ is a constant.
In addition to the non-negativity of the initial conditions $g_1(x)$ and $g_2(x)$, we impose the additional conditions:
\begin{itemize}
  \item [(H2)] $2\int_{\Omega}\,F(g_1(x),g_2(x))\,dx\geq \gamma_1\,\int_{\partial\Omega}\, g_1^2\,ds+\int_{\Omega}\, |\nabla g_1|^2\,dx$;
  \item [(H3)] $2\int_{\Omega}\,F(g_1(x),g_2(x))\,dx\geq \gamma_2\,\int_{\partial\Omega}\, g_2^2\,ds+\int_{\Omega}\, |\nabla g_2|^2\,dx$.
\end{itemize}


Under these conditions we have:

\vspace{5mm}
\begin{thm}\label{thm: main thm}
Suppose that $(u(x,t),v(x,t))$ is a pair of solution to \eqref{eqn: IBV prob for RD system}. If $(H1)\sim(H3)$ are satisfied, then at least one of $u(x,t)$ and $v(x,t)$ blows up in finite time $t^{\ast}$, where $t^{\ast}$ is bounded above by \eqref{eqn: lower bdd of blow up time}.
\end{thm}

The physical meaning of the Robin boundary conditions can be explained as follows. Suppose that $u$ and $v$ represent temperature, and are governed by the equations in the problem \eqref{eqn: IBV prob for RD system}. Then the Robin boundary conditions mean that the heat flux $\frac{\displaystyle\partial u}{\displaystyle\partial \nu}$ and $\frac{\displaystyle\partial v}{\displaystyle\partial \nu}$ on the boundary of $\Omega$ are proportional to the temperature $u$ and $v$ on the boundary of $\Omega$, respectively. Since $\gamma_1$ and $\gamma_2$ are positive constants, it follows that the larger the heat flux is, the smaller the temperature is. We note that, from the biological point of view, the temperature and the heat flux can be substituted respectively to population density and population flux. 

In other words, the larger the population flux is, the smaller the population density is. As a consequence, when the population flux on $\partial\Omega$ is large, the population density on $\partial\Omega$ is small. The low density of $u$ and $v$ on $\partial\Omega$ then may result in the blow-up of $u$ or $v$ since the large flux flows into $\Omega$ but on the boundary of $\Omega$, the density of $u$ and $v$ are restricted to be small. Therefore, $u$ and $v$ are may be forced to aggregate together so that blow-up occurs.

The remainder of this paper is organized as follows. In Section~\ref{sec: Proof of Theorem}, we give the proof of Theorem~\ref{thm: main thm}. This theorem asserts that for certain initial conditions, the solutions of \eqref{eqn: IBV prob for RD system} blow up in finite time, with an upper bound given by \eqref{eqn: lower bdd of blow up time}. Section~\ref{sec: Lower bound for blow-up time} is devoted to determining a lower bound on the blow-up time when blow-up does occur. In addition, Theorem~\ref{thm: main thm for lower bound for blow-up time (H2) and (H3) are replaced} provides a cooperative system considered in \cite{Bedjaoui&Souplet-blowup-absorption02} a lower bounded for blow-up time when blow-up does occur. Finally, we conclude the present paper with some remarks in Section~\ref{sec: Concluding Remarks}.

\vspace{5mm}
\setcounter{equation}{0}
\setcounter{figure}{0}
\section{Proof of Theorem~\ref{thm: main thm}}\label{sec: Proof of Theorem}
\vspace{5mm}

In this section, Theorem~\ref{thm: main thm} is proven.

\vspace{5mm}
\begin{proof}
First of all, we define

\begin{equation}
E(t)=\int_{\Omega}\,(u^2+v^2)\,dx.
\end{equation}
By means of integration by parts and $(H1)$, we arrive at
\begin{align}
  E'(t)=& 2\,\int_{\Omega}\,u\,(\Delta u+f_1(u,v))+v\,(\Delta v+f_2(u,v))\,dx\notag\\
     =& -2\,\gamma_1\,\int_{\partial\Omega}\,u^2\,ds-2\,\int_{\Omega}\,|\nabla u|^2\,dx
        +2\,\int_{\Omega}\,u\,f_1(u,v)\,dx\notag\\
      & -2\,\gamma_2\,\int_{\partial\Omega}\,v^2\,ds-2\,\int_{\Omega}\,|\nabla v|^2\,dx
        +2\,\int_{\Omega}\,v\,f_2(u,v)\,dx\notag\\
  \geq&  -2\,(1+\alpha)\,(\gamma_1\,\int_{\partial\Omega}\,u^2\,ds+\int_{\Omega}\,|\nabla u|^2\,dx)
         -2\,(1+\alpha)\,(\gamma_2\,\int_{\partial\Omega}\,v^2\,ds+\int_{\Omega}\,|\nabla v|^2\,dx)\notag\\
      &+ 4\,(1+\alpha)\,\int_{\Omega}\,F(u,v)\,dx
       .\notag\\
\end{align}
Letting
\begin{align}
  J(t)
     =&  -2\,(1+\alpha)\,(\gamma_1\,\int_{\partial\Omega}\,u^2\,ds+\int_{\Omega}\,|\nabla u|^2\,dx)
         -2\,(1+\alpha)\,(\gamma_2\,\int_{\partial\Omega}\,v^2\,ds+\int_{\Omega}\,|\nabla v|^2\,dx)\notag\\
      &  +4\,(1+\alpha)\,\int_{\Omega}\,F(u,v)\,dx,\notag\\
\end{align}
we calculate the derivative of $J(t)$ to obtain
\begin{align}
J'(t)=&-4\,(1+\alpha)(\gamma_1\,\int_{\partial\Omega}\,u\,u_t\,ds
              +\int_{\Omega}\,\nabla u\cdot \nabla u_t\,dx-\int_{\Omega}\,f_1(u,v)\,u_t\,dx)\notag\\
         &-4\,(1+\alpha)(\gamma_2\,\int_{\partial\Omega}\,v\,v_t\,ds
              +\int_{\Omega}\,\nabla v\cdot \nabla v_t\,dx-\int_{\Omega}\,f_2(u,v)\,v_t\,dx)\notag\\
         =&-4\,(1+\alpha)(\gamma_1\,\int_{\partial\Omega}\,u\,u_t\,ds
              +\int_{\partial\Omega}\,\frac{\partial u}{\partial \nu}\,u_t\,ds
              -\int_{\Omega}\,u_t\,(\Delta u+f_1(u,v))\notag\\
         &-4\,(1+\alpha)(\gamma_2\,\int_{\partial\Omega}\,v\,v_t\,ds
              +\int_{\partial\Omega}\,\frac{\partial v}{\partial \nu}\,v_t\,ds
              -\int_{\Omega}\,v_t\,(\Delta v+f_2(u,v))\notag\\
     \geq& 4\,(1+\alpha)\,\int_{\Omega}\,(u_t^2+v_t^2)\,dx\notag,\\
\end{align}
by virtue of integration by parts and the Robin boundary conditions in \eqref{eqn: IBV prob for RD system}. Due to $(H2)$ and $(H3)$, we have $J(0)\geq0$. Then $J(0)\geq0$ and $J'(t)\geq 0$, for $t\geq0$, yield $J(t)\geq 0$  for $t\geq0$. Because
\begin{equation}
E'(t)=2\int_{\Omega}\,(u\,u_t+v\,v_t)\,dx,
\end{equation}
we have by applying Cauchy-Schwartz inequality
\begin{align}\label{eqn: ineq by Cauchy-Schwartz inequality }
(E'(t))^2=& 4\,\bigg(   \Big(\int_{\Omega}\,u\,u_t\,dx\Big)^2
                              +\Big(\int_{\Omega}\,v\,v_t\,dx\Big)^2
                              +  2\,\Big(\int_{\Omega}\,u\,u_t\,dx\Big)
                                  \,\Big(\int_{\Omega}\,v\,v_t\,dx\Big)  \bigg)\notag\\
            \leq& 4\,\bigg( \int_{\Omega}\,u^2\,dx \int_{\Omega}\,u_t^2\,dx
                           +\int_{\Omega}\,v^2\,dx \int_{\Omega}\,v_t^2\,dx     \notag\\
                           &+ 2\,\Big(   \int_{\Omega}\,u^2\,dx \,\int_{\Omega}\,u_t^2\,dx
                                         \int_{\Omega}\,v^2\,dx \,\int_{\Omega}\,v_t^2\,dx \Big)^{\frac{1}{2}}   \bigg)\notag\\
     \leq& 4\,\int_{\Omega}\,(u^2+v^2)\,dx \int_{\Omega}\,(u_t^2+v_t^2)\,dx\notag\\
     \leq& \frac{1}{1+\alpha}\,E(t)\,J'(t).
\end{align}
The last inequality holds since $\sqrt{a\,b}\leq \frac{a+b}{2}$, for $a,b\geq0$. From the definition of $J(t)$, $E'(t)\geq J(t)$, and consequently we can replace $E'(t)$ by $J(t)$ in \eqref{eqn: ineq by Cauchy-Schwartz inequality } to obtain
\begin{equation}
J(t)\,E'(t)\leq \frac{1}{1+\alpha}\,E(t)\,J'(t)
\end{equation}
or
\begin{equation}
(1+\alpha)\,\frac{E'(t)}{E(t)}\leq \frac{J'(t)}{J(t)}.
\end{equation}
Following the same arguments in \cite{Payne&SchaeferRobin08}, we obtain that $\varphi(t)$ satisfies the inequality
\begin{equation}\label{eqn: ineq to have a contradiction}
\frac{1}{(E(t))^{\alpha}}\leq \frac{1}{(E(0))^{\alpha}}-
\alpha\,M\,t,
\end{equation}
where $M=\frac{J(0)}{(E(0))^{1+\alpha}}$. Since the last inequality cannot be true for all $t\geq0$, we infer that at least one of $u(x,t)$ and $v(x,t)$ blows up in finite time, $t^{\ast}$, where $t^{\ast}$ is bounded above by
\begin{equation}\label{eqn: lower bdd of blow up time}
t^{\ast}\leq\frac{1}{\alpha\,M\,(E(0))^{\alpha}}.
\end{equation}
\end{proof}
\vspace{5mm}
We note that Theorem~\ref{thm: main thm} remains true if the Robin boundary conditions are replaced by Neumann boundary conditions (i.e. $\gamma_1=\gamma_2=0$).
To illustrate the results in Theorem~\ref{thm: main thm}, we give an example of the nonlinearity in \eqref{eqn: IBV prob for RD system}. Take $F(u,v)=u^2\,v^3$, then $f_1=\partial_u F(u,v)=2\,u\,v^3$, $f_2=\partial_v F(u,v)=3\,u^2\,v^2$,
\begin{equation}
u\,f_1(u,v)+v\,f_2(u,v)=5\,u^2\,v^3,
\end{equation}
and
\begin{equation}
2(1+\alpha)\,F(u,v)=2(1+\alpha)\,u^2\,v^3.
\end{equation}
Clearly, $(H1)$ is fulfilled when $0<\alpha\leq\frac{3}{2}$.

\vspace{5mm}
\begin{rem}
Equivalently, $(H1)$ can be rewritten as
\begin{equation}
u\,\partial_u F(u,v)+v\,\partial_v F(u,v)\geq 2(1+\alpha)\,F(u,v).
\end{equation}
For equality, that is
\begin{equation}
u\,\partial_u F(u,v)+v\,\partial_v F(u,v)=2(1+\alpha)\,F(u,v),
\end{equation}
which can be solved by \textit{the method of characteristics} to give
\begin{equation}
F(u,v)=c\,u^{2(1+\alpha)}\,h\bigg(\frac{v}{u}\bigg),
\end{equation}
where $h=h(w)$ is an arbitrary smooth function and $c$ is an arbitrary constant. This solution is useful in looking for nonlinearities $f_1$ and $f_2$ in \eqref{eqn: IBV prob for RD system} which satisfy $(H1)$. 
\end{rem}

\vspace{5mm}
\setcounter{equation}{0}
\setcounter{figure}{0}
\section{Lower bound for blow-up time}\label{sec: Lower bound for blow-up time}
\vspace{5mm}

In this section, a lower bound on the blow-up time is obtained when blow-up does occur.
In particular, (2.16) in \cite{Payne&Schaefer-blowup-lowerbdd-Neumann06} plays an essential role in proving the following 

\vspace{5mm}
\begin{thm}\label{thm: main thm for lower bound for blow-up time}
Let $(u(x,t),v(x,t))$ be a pair of nonnegative solutions to \eqref{eqn: IBV prob for RD system} and at least one of $u(x,t)$ and $v(x,t)$ blow up in finite time, $t=t^{\ast}$. Suppose that $(A1)\sim(A3)$ below are also satisfied:
\begin{itemize}
  \item [(A1)] $\Omega\in\mathbb{R}^3$ is a bounded smooth convex domain;
  \item [(A2)] $f_1(u,v)\leq k_1\,u^{p+1}$ for $k_1,u,v>0$ and $p\geq1$;
  \item [(A3)] $f_2(u,v)\leq k_2\,v^{p+1}$ for $k_2,u,v>0$ and $p\geq1$.
\end{itemize}
Then $t^{\ast}$ is bounded below by \eqref{eqn: lower bdd of blow up time-2} .
\end{thm}
\begin{proof}
First let us define the auxiliary function
\begin{equation}
\mathcal{E}(t)=\int_{\Omega}\,(u^{2p}+v^{2p})\,dx.
\end{equation}
On applying Green's first identity and the equality $u^{2p-2}\,|\nabla u|^2=p^{-2}\,|\nabla u^p|^2$ , we arrive at
\begin{align}
  \mathcal{E}'(t)=& 2p\int_{\Omega}\,u^{2p-1}\,(\Delta u+f_1(u,v))+v^{2p-1}\,(\Delta v+f_2(u,v))\,dx\notag\\
     =& -2p\,\gamma_1\int_{\partial\Omega}\,u^{2p}\,ds-2p\,(2p-1)\,\int_{\Omega}\,u^{2p-2}\,|\nabla u|^2\,dx
        +2p\,\int_{\Omega}\,u^{2p-1}\,f_1(u,v)\,dx\notag\\
      & -2p\,\gamma_2\int_{\partial\Omega}\,v^{2p}\,ds-2p\,(2p-1)\,\int_{\Omega}\,v^{2p-2}\,|\nabla v|^2\,dx
        +2p\,\int_{\Omega}\,v^{2p-1}\,f_2(u,v)\,dx\notag\\
  \leq& -2\,(2p-1)\,p^{-1}\int_{\Omega}\,|\nabla u^{p}|^2\,dx
        +2p\,\int_{\Omega}\,u^{2p-1}\,f_1(u,v)\,dx\notag\\
      & -2\,(2p-1)\,p^{-1}\int_{\Omega}\,|\nabla v^{p}|^2\,dx
        +2p\,\int_{\Omega}\,v^{2p-1}\,f_2(u,v)\,dx\notag\\
  \leq& -2\,(2p-1)\,p^{-1}\int_{\Omega}\,|\nabla u^{p}|^2\,dx
        +2p\,k_1\,\int_{\Omega}\,u^{3p}\,dx\notag\\
      & -2\,(2p-1)\,p^{-1}\int_{\Omega}\,|\nabla v^{p}|^2\,dx
        +2p\,k_2\,\int_{\Omega}\,v^{3p}\,dx\notag\\
\end{align}
since $(A2)$ and $(A3)$ hold. Now our strategy is to relate $\int_{\Omega}\,v^{3p}\,dx$ in terms of $\mathcal{E}(t)$ and $\int_{\Omega}\,|\nabla v^{p}|^2\,dx$. To this end, the integral inequality (see (2.16) in \cite{Payne&Schaefer-blowup-lowerbdd-Neumann06}) is used:
\begin{equation}
\int_{\Omega}\,u^{3p}\,dx\leq
\frac{1}{3^{\frac{3}{4}}}\,\Bigg\{
\frac{3}{2\rho}\int_{\Omega}\,u^{2p}\,dx+
\bigg(\frac{d}{\rho}+1\bigg)
\bigg(\int_{\Omega}\,u^{2p}\,dx\bigg)^{\frac{1}{2}}
\bigg(\int_{\Omega}\,|\nabla u^{p}|^2\,dx\bigg)^{\frac{1}{2}}
\Bigg\}^{\frac{3}{2}},
\end{equation}
where for some origin inside $\Omega$,
\begin{equation}
\rho=\min_{\partial\Omega}\, x_i\,\nu_i, \quad d^2=\max_{\bar{\Omega}}\, x_i\,x_i,
\end{equation}
for $\nu_i$ the $i-$th component of the unit outward normal to $\partial\Omega$. Thus,
\begin{align}
  \mathcal{E}'(t)
  \leq& -2\,(2p-1)\,p^{-1}\int_{\Omega}\,|\nabla u^{p}|^2\,dx
        -2\,(2p-1)\,p^{-1}\int_{\Omega}\,|\nabla v^{p}|^2\,dx\notag\\
      & +\frac{2p\,k_1}{3^{\frac{3}{4}}}\,\Bigg\{
         \frac{3}{2\rho}\int_{\Omega}\,u^{2p}\,dx+
         \bigg(\frac{d}{\rho}+1\bigg)
         \bigg(\int_{\Omega}\,u^{2p}\,dx\bigg)^{\frac{1}{2}}
         \bigg(\int_{\Omega}\,|\nabla u^{p}|^2\,dx\bigg)^{\frac{1}{2}}
         \Bigg\}^{\frac{3}{2}}\notag\\
      & +\frac{2p\,k_2}{3^{\frac{3}{4}}}\,\Bigg\{
         \frac{3}{2\rho}\int_{\Omega}\,v^{2p}\,dx+
         \bigg(\frac{d}{\rho}+1\bigg)
         \bigg(\int_{\Omega}\,v^{2p}\,dx\bigg)^{\frac{1}{2}}
         \bigg(\int_{\Omega}\,|\nabla v^{p}|^2\,dx\bigg)^{\frac{1}{2}}
         \Bigg\}^{\frac{3}{2}}\notag.\\
\end{align}
By virtue of the elementary inequalities $(a+b)^{\frac{3}{2}}\leq 2^{\frac{1}{2}}(a^{\frac{3}{2}}+b^{\frac{3}{2}})$ and $a^{\frac{1}{4}}\,b^{\frac{3}{4}}\leq \frac{1}{4}a+\frac{3}{4}b$, we obtain
\begin{align}
  \mathcal{E}'(t)
  \leq& -2\,(2p-1)\,p^{-1}\int_{\Omega}\,|\nabla u^{p}|^2\,dx
        -2\,(2p-1)\,p^{-1}\int_{\Omega}\,|\nabla v^{p}|^2\,dx\notag\\
      & +\frac{2p\,k_1}{3^{\frac{3}{4}}}\,2^{\frac{1}{2}}\,\Bigg\{
         \bigg(\frac{3}{2\rho}\bigg)^{\frac{3}{2}}
         \bigg(\int_{\Omega}\,u^{2p}\,dx\bigg)^{\frac{3}{2}}+
         \bigg(\frac{d}{\rho}+1\bigg)^{\frac{3}{2}}
         \bigg(\int_{\Omega}\,u^{2p}\,dx\bigg)^{\frac{3}{4}}
         \bigg(\int_{\Omega}\,|\nabla u^{p}|^2\,dx\bigg)^{\frac{3}{4}}
         \Bigg\}\notag\\
      & +\frac{2p\,k_2}{3^{\frac{3}{4}}}\,2^{\frac{1}{2}}\,\Bigg\{
         \bigg(\frac{3}{2\rho}\bigg)^{\frac{3}{2}}
         \bigg(\int_{\Omega}\,v^{2p}\,dx\bigg)^{\frac{3}{2}}+
         \bigg(\frac{d}{\rho}+1\bigg)^{\frac{3}{2}}
         \bigg(\int_{\Omega}\,v^{2p}\,dx\bigg)^{\frac{3}{4}}
         \bigg(\int_{\Omega}\,|\nabla v^{p}|^2\,dx\bigg)^{\frac{3}{4}}
         \Bigg\}\notag\\
  \leq& -2\,(2p-1)\,p^{-1}\int_{\Omega}\,|\nabla u^{p}|^2\,dx
        -2\,(2p-1)\,p^{-1}\int_{\Omega}\,|\nabla v^{p}|^2\,dx\notag\\
      & +\frac{2p\,k_1}{3^{\frac{3}{4}}}\,2^{\frac{1}{2}}\,\Bigg\{
         \bigg(\frac{3}{2\rho}\bigg)^{\frac{3}{2}}
         \bigg(\int_{\Omega}\,u^{2p}\,dx\bigg)^{\frac{3}{2}}\notag\\
      & \hspace{25mm} +
         \bigg(\frac{d}{\rho}+1\bigg)^{\frac{3}{2}}
         \Bigg[
         \frac{\beta_1^{-3}}{4}\bigg(\int_{\Omega}\,u^{2p}\,dx\bigg)^{3}
         +
         \frac{3\beta_1}{4}\bigg(\int_{\Omega}\,|\nabla u^{p}|^2\,dx\bigg)^{\frac{3}{4}}
         \Bigg]
         \Bigg\}\notag\\
      & +\frac{2p\,k_2}{3^{\frac{3}{4}}}\,2^{\frac{1}{2}}\,\Bigg\{
         \bigg(\frac{3}{2\rho}\bigg)^{\frac{3}{2}}
         \bigg(\int_{\Omega}\,v^{2p}\,dx\bigg)^{\frac{3}{2}}\notag\\
      &  \hspace{25mm} +
         \bigg(\frac{d}{\rho}+1\bigg)^{\frac{3}{2}}
         \Bigg[
         \frac{\beta_2^{-3}}{4}\bigg(\int_{\Omega}\,v^{2p}\,dx\bigg)^{3}
         +
         \frac{3\beta_2}{4}\bigg(\int_{\Omega}\,|\nabla v^{p}|^2\,dx\bigg)^{\frac{3}{4}}
         \Bigg]
         \Bigg\}\notag,\\
\end{align}
where $\beta_1$ and $\beta_2$ are positive constants satisfying
\begin{equation}
-2\,(2p-1)\,p^{-1}+\frac{3^\frac{1}{4}p\,k_1}{2^\frac{1}{2}}\bigg(\frac{d}{\rho}+1\bigg)^{\frac{3}{2}}\,\beta_1\leq0,
\quad
-2\,(2p-1)\,p^{-1}+\frac{3^\frac{1}{4}p\,k_2}{2^\frac{1}{2}}\bigg(\frac{d}{\rho}+1\bigg)^{\frac{3}{2}}\,\beta_1\leq0.
\end{equation}
Accordingly,
\begin{align}
  \mathcal{E}'(t)
  \leq& \frac{2p\,k_1}{3^{\frac{3}{4}}}\,2^{\frac{1}{2}}\,\Bigg\{
         \bigg(\frac{3}{2\rho}\bigg)^{\frac{3}{2}}
         \bigg(\int_{\Omega}\,u^{2p}\,dx\bigg)^{\frac{3}{2}}
         +\bigg(\frac{d}{\rho}+1\bigg)^{\frac{3}{2}}
         \Bigg[
         \frac{\beta_1^{-3}}{4}\bigg(\int_{\Omega}\,u^{2p}\,dx\bigg)^{3}
         \Bigg]
         \Bigg\}+\notag\\
      & \frac{2p\,k_2}{3^{\frac{3}{4}}}\,2^{\frac{1}{2}}\,\Bigg\{
         \bigg(\frac{3}{2\rho}\bigg)^{\frac{3}{2}}
         \bigg(\int_{\Omega}\,v^{2p}\,dx\bigg)^{\frac{3}{2}}
         +\bigg(\frac{d}{\rho}+1\bigg)^{\frac{3}{2}}
         \Bigg[
         \frac{\beta_2^{-3}}{4}\bigg(\int_{\Omega}\,v^{2p}\,dx\bigg)^{3}
         \Bigg]
         \Bigg\},\\
  \leq& 3^\frac{3}{4}\,p\,k\,\rho^{-\frac{3}{2}}\,\Bigg (\int_{\Omega}\,(u^{2p}+v^{2p})\,dx \Bigg)^{\frac{3}{2}}+
        \frac{p\,k}{2^\frac{1}{2}\,3^{\frac{3}{4}}}\,\bigg(\frac{d}{\rho}+1\bigg)^{\frac{3}{2}}\,\beta^{-3}\,
        \bigg(\int_{\Omega}\,u^{2p}+v^{2p}\,dx\bigg)^{3},
\end{align}
where $k=\max(k_1,k_2)$ and  $\beta=\min(\beta_1,\beta_2)$. As a result, we obtain
\begin{equation}
\mathcal{E}'(t)\leq K_1\,\mathcal{E}^{\frac{3}{2}}(t)+K_2\,\mathcal{E}^3(t),
\end{equation}
where
\begin{equation}
K_1=3^\frac{3}{4}\,p\,k\,\rho^{-\frac{3}{2}},\quad K_2=\frac{p\,k}{2^\frac{1}{2}\,3^{\frac{3}{4}}}\,\bigg(\frac{d}{\rho}+1\bigg)^{\frac{3}{2}}\,\beta^{-3}.
\end{equation}
Integrating yields
\begin{equation}\label{eqn: lower bdd of blow up time-2}
t^{\ast}\geq \int_{\mathcal{E}(0)}^{\infty}\,\frac{d\xi}{K_1\,\xi^{\frac{3}{2}}+K_2\,\xi^3}.
\end{equation}
This completes the proof of the theorem.
\end{proof}

\vspace{5mm}
We remark that Theorem~\ref{thm: main thm for lower bound for blow-up time} remains true if the Robin boundary conditions are replaced by Neumann boundary conditions. From the proof of Theorem~\ref{thm: main thm for lower bound for blow-up time}, it is readily seen that the following result is true under an assumption which is weaker than $(A2)$ and $(A3)$.

\vspace{5mm}

\begin{thm}\label{thm: main thm for lower bound for blow-up time (H2) and (H3) are replaced}
Let $(u(x,t),v(x,t))$ be a pair of nonnegative solution of \eqref{eqn: IBV prob for RD system} and that at least one of $u(x,t)$ and $v(x,t)$ blow up in finite time, $t=t^{\ast}$. Suppose that $(A1)'\sim(A2)'$ below are satisfied:
\begin{itemize}
  \item [(A1)'] $\Omega\in\mathbb{R}^3$ is a bounded smooth convex domain;
  \item [(A2)'] For $k_1,k_2,u,v>0$ and $p\geq1$,
                \begin{equation}\label{eqn: (Hi) which replaces (H2) and (H3)}
                 u^{2p-1}\,f_1(u,v)+v^{2p-1}\,f_2(u,v)\leq k_1\,u^{3p}+k_2\,v^{3p}.
                \end{equation}

\end{itemize}
Then $t^{\ast}$ is bounded below by \eqref{eqn: lower bdd of blow up time-2} .
\end{thm}
\vspace{5mm}

Now we are in the position to apply the following results to give an example of the above theorem.

\vspace{5mm}
\begin{thm}[From \cite{Bedjaoui&Souplet-blowup-absorption02}]\label{thm: blowup soln in RD system with absorption}
Consider the following reaction-diffusion system with \textit{absorption}:
\begin{equation}\label{eqn: IBV prob for RD system with absorption}
\begin{cases}
\vspace{2mm}
u_t= \Delta u+v^p-a\,u^r, \quad \text{in}\quad \Omega\times(0,\infty), \\
\vspace{2mm}
v_t= \Delta v+u^q-b\,v^s, \quad \text{in}\quad \Omega\times(0,\infty),\\
\vspace{2mm}
u=0, \quad
v=0, \quad  \text{on}\quad \partial\Omega\times(0,\infty),\\
u(x,0)=g_1(x), \quad v(x,0)=g_2(x) \quad \text{in}\quad \Omega,
\end{cases}
\end{equation}
where $\Omega\in\mathbb{R}^n$; $a$, $b$, $p$, $q$, $r$ and $s$ are positive constants; $g_1$ and $g_2$ are nonnegative functions.
\begin{itemize}
  \item [(i)] If $p\,q>\max(r,1)\max(s,1)$, then there exist solutions of \eqref{eqn: IBV prob for RD system with absorption} which blow up in finite time.
  \item [(ii)] If $p\,q<\max(r,1)\max(s,1)$, then all solutions of \eqref{eqn: IBV prob for RD system with absorption} are global. Moreover, if $r$,$s\geq1$ (hence $p\,q<r\,s$), they are uniformly bounded.
  \item [(iii)] If $p\,q=\max(r,1)\max(s,1)$, then
                \begin{itemize}
                  \item [(a)] if $r,s>1$ and $a$ and $b$ are sufficiently small, then there exist solutions of \eqref{eqn: IBV prob for RD system with absorption} which blow up in finite time;
                  \item [(b)] if $r,s\geq1$, and $a^q\,b^r\geq1$ (equivalently, $a^s\,b^p\geq1$), then all solutions are global and uniformly bounded;
                  \item [(c)] if $r$ or $s\leq1$, then all solutions are global (possibly bounded).
                \end{itemize}
\end{itemize}
\end{thm}
\vspace{5mm}

For the nonlinearity $f_1=v^3-a\,u^3$ and $f_1=u^3-b\,v^3$ (by choosing $p=q=r=s=3$ in Theorem~\ref{thm: blowup soln in RD system with absorption}), we let $p=2$ in Theorem~\ref{thm: main thm for lower bound for blow-up time (H2) and (H3) are replaced} so that
\begin{equation}
u^{2p-1}\,f_1(u,v)+v^{2p-1}\,f_2(u,v)=2\,u^3v^3-a\,u^6-b\,v^6,
\end{equation}
while
\begin{equation}
k_1\,u^{3p}+k_2\,v^{3p}=k_1\,u^{6}+k_2\,v^{6}.
\end{equation}
By choosing $k_1,k_2\geq1$, it is easy to see that \eqref{eqn: (Hi) which replaces (H2) and (H3)} automatically holds for any $a,b>0$. We apply Theorem~\ref{thm: blowup soln in RD system with absorption} $(iii)-(a)$ to conclude that \eqref{eqn: IBV prob for RD system with absorption} with $p=q=r=s=3$ and $a,b$ sufficiently small admits a blow-up solution. Now Theorem~\ref{thm: main thm for lower bound for blow-up time (H2) and (H3) are replaced} provides a lower bounded for blow-up time $t=t^{\ast}$.

\vspace{5mm}
\setcounter{equation}{0}
\setcounter{figure}{0}
\section{Concluding Remarks}\label{sec: Concluding Remarks}
\vspace{5mm}

In Theorem~\ref{thm: main thm} and Theorem~\ref{thm: main thm for lower bound for blow-up time}, the estimates of the blow-up time for the blow-up solutions of \eqref{eqn: IBV prob for RD system} are given by means of the differential inequality technique adapted from \cite{Payne&SchaeferRobin08}. In the present paper, we have shown that this technique can also be applied to certain reaction-diffusion \textit{systems}. From the proofs of Theorem~\ref{thm: main thm} and Theorem~\ref{thm: main thm for lower bound for blow-up time} however, we cannot determine whether one of $u$ and $v$ blows up or if both $u$ and $v$ blow up. This problem is left for future work.

\vspace{5mm}

\textbf{Acknowledgments.} The authors wish to express sincere gratitude to Dr. Tom Mollee and Dr. Ya-Yu Chen for their careful reading of the manuscript and valuable suggestions and comments to improve the readability and accuracy of the paper. Special thanks are due to Professor Masayasu Mimura for the valuable discussion.

\vspace{5mm}









\end{document}